\documentclass{article}
\usepackage[utf8]{inputenc}
\usepackage{latexsym}
\usepackage{amsmath}
\usepackage{amsfonts}
\usepackage{graphicx}
\usepackage{color}
\usepackage{hyperref}
\bgroup\catcode`\#=12\egroup 

\bgroup\catcode`\#=12\egroup 

\bgroup\catcode`\#=12\egroup

\bgroup\catcode`\#=12\egroup

\bgroup\catcode`\#=12\egroup

\bgroup\catcode`\#=12\egroup 

\usepackage{theorem}

\theoremheaderfont{\scshape}
\newtheorem{theorem}{Theorem}[section]

\newtheorem{lemma}[theorem]{Lemma}
\newtheorem{prop1}[theorem]{Proposition}
\theorembodyfont{\normalfont}
\newtheorem{remark}{Remark}[section]
\newtheorem{definition}{Definition}[section]
\newenvironment{proof}{{\bf Proof }}{\hbox{~} \hfill \rule{0.5em}{0.5em}\\}
\numberwithin{equation}{section}

\begin{document}
\title{On shape and topological optimization problems with constraints  Helmholtz equation and spectral problems
\footnote{This work is supported by  NLAGA Project (Non Linear Analysis, Geometry and Applications Project(http://nlaga-simons.ucad.sn/index.php/nlaga-soft/14-home)) and Deutsche Forschungsgemeinschaft within the Priority program SPP 1962 "Non-smooth and Complementarity based Distributed Parameter Systems: Simulation and Hierarchical Optimization". The authors would like to thank Volker Schulz ( University Trier, Trier, Germany) and Luka Schlegel
(University Trier, Trier, Germany) for helpful and interesting discussions within the project Shape
Optimization Mitigating Coastal Erosion (SOMICE).}}         
\date{}          

\maketitle
\centerline{\scshape Mame Gor Ngom \footnote{mamegor.ngom@uadb.edu.sn}}
\medskip
{\footnotesize
 \centerline{Universit\'e Alioune Diop de Bambey, }
  \centerline{  Ecole Doctorale des Sciences et Techniques et Sciences de la Soci\'et\'e. }
   \centerline{Laboratoire de Math\'ematiques de la D\'ecision et d'Analyse Num\'erique}
    \centerline{ (L.M.D.A.N) F.A.S.E.G)/F.S.T. }
}

\medskip
\centerline{\scshape Ibrahima Faye \footnote{ibrahima.faye@uadb.edu.sn}, Diaraf Seck \footnote{diaraf.seck@ucad.edu.sn }}
\medskip
{\footnotesize
 \centerline{Universit\'e Alioune Diop de Bambey, UFR S.A.T.I.C, BP 30 Bambey (S\'en\'egal),}
\centerline{  Ecole Doctorale des Sciences et Techniques et Sciences de la Soci\' et\'e. }
   \centerline{Laboratoire de Math\'ematiques de la D\'ecision et d'Analyse Num\'erique}
\centerline{ (L.M.D.A.N). }
}

\pagestyle{myheadings}
 \renewcommand{\sectionmark}[1]{\markboth{#1}{}}
\renewcommand{\sectionmark}[1]{\markright{\thesection\ #1}}

\begin{abstract}\noindent 
Coastal erosion describes the displacement of sand caused by the movement induced by tides, waves or currents.  Some of its wave phenomena are modeled by Helmholtz-type equations. Our purposes, in this paper are, first, to study optimal shapes obstacles to mitigate sand transport under the constraint of the Helmholtz equation.  And the second side of this work is related to Dirichlet and Neumann  spectral problems.We show the existence of optimal shapes in a general admissible set of  quasi open sets. And necessary  optimality conditions of first order are given in a regular framework.
\\

\end{abstract}

\section{Introduction}
Coastal erosion is and will be a major and growing environmental problem.
Several phenomena contribute to the significant advance of the sea. These include climate change with the rise in sea level due to the melting of ice at the Earth's poles, the amplification of the tidal effect which leads to the transport of large masses of sand, storms etc. It should be noted in passing that storm surges are the immediate consequence of local weather conditions and can inflict serious damage to human life and property by flooding the coast. 
For a very long time, at least forty years or more, the study of this phenomenon has occupied the minds of scientists (hydro-geologists, climatologists, geologists, hydraulic engineers, mechanics, to name but a few). To understand these aspects, modelling of these phenomena is more than necessary.
A significant advance in the field of wave modelling was made by Berkhoff \cite{Ber1, Ber2}, who gave a refraction-diffraction type equation. It can be used for a wide range of ocean wave frequencies, since it passes, in the limit, to the deep and shallow water equations. The usefulness of the Berkhoff equation for obtaining good simulations of wave behaviour in a wide variety of different situations has been demonstrated by many researchers. We refer the interested reader to the following references : \cite{IAMB, PCP} and the references contained in these two articles.
To mitigate waves and sand transport, obstacles (groins, breakwaters, jetties and many other structures) are designed, built and constructed at the coast. But their geometric shapes and locations are usually determined by simple hydrodynamic assumptions, structural strength laws and empirical considerations.  It should be noted that in coastal engineering it is crucial to minimise wave reflection along the coast. And the use of emergent structures is an interesting and even efficient way to achieve this. 
Many types of waves involving different physical factors exist in the ocean. As in the elementary problem of a spring-mass system, all waves must be associated with a certain restoring force. 
But in this work, we consider a linear incident wave following Airy's theory for small amplitude waves. The scattering towards the structure is modeled by the Helmholtz equation which can be obtained from the Berkhoff equation see \cite{Ber1}. And, the objective is to study a minimal action under the constraint of the Helmholtz boundary problem with the geometry and the topology of the emerging obstacle or structure as the main variable.
In structural shape optimization problems, the aim is to improve the performance of the structure by modifying its boundaries. This can be numerically achieved by minimizing an objective function subjected to certain constraints \cite{Chara}.
All functions are related to the design variables, which have some of the coordinates of the key points in the boundary of the structure. The shape optimization methodology proceeds with the following two steps: the parametric and the non parametric methods as describe in Meske and al.  \cite {Meske}, Murat and Simon \cite{Murat} Pironneau \cite{Pironneau} and Sokolowski and Zolesio[  \cite{Sokolowski} and the topological optimization \cite{Sokolowski, Masmoudi}. 

 In these last time, many authors have been interested in the open  question of finding solutions to coastal erosion. Most of them offer models obtained empirically \cite{Idier} or by developing mathematical models \cite{FaFreSe}, see also  references therein these two papers. But the question still remains. In this work, we seek to better understand the causes of this phenomenon by using shape and topological optimization methods. 
Our approach is to use shape and topology optimization tools to  propose a qualitative study of coastal structures such as breakwaters, groins and other innovative shapes. Firstly, in order to seek a certain completeness of this work, we study an example of a shape functional problem under constraints of Helmholtz type equations or eigenvalue problems  with Dirichlet type boundary conditions. Then we look for optimality conditions according to several configurations: the case without eigenvalues, the situation where we are facing a spectral problem with single or multiple eigenvalues. This step will be followed by topological sensitivity calculations. In particular, we establish calculations of the solutions of the Helmholtz boundary problems describing the water waves scattered by the structure. This modifies its shape accordingly, in order to minimize a predefined cost function such as the energy of the water waves.
Let us quote that, these types of shape optimization and topological  problems associating a functional and a Helmholtz-type PDE or spectral problems  as constraint equation have already been studied in the literature, see for instance the works  due to  B. Samet et al.  \cite{Samet2}, S. Amstutz \cite{Amstutz, Amstutz1} and A. Novotny et al. \cite{NovotnySokoloZo}.
 And this paper is  a trial to gather various results, in one paper,  that had been already established in different works. In addition, we attempt to complete some interesting parts or questionnings related to our subject on coastal erosion.\\
The paper is organized as follows. In section 2, we give problem formulation and existence results of  optimization problems. In section 3, the shape derivatives of the functionals are given. The final section, is devoted to topological approach and a link between topological and shape derivative.

\section{Existence result of shape optimization problem}
\noindent
In this section, we consider several class of functional defined on an admissible set.
\subsection{Problem formulation}
Let $\Omega\subset\mathbb R^N,\,\,N\geq 2$ be an open domain with boundary $\partial\Omega.$ Let's consider, in this subsection, the following shapes functionals 
 \begin{equation}\label{fonct0}
J(\Omega)=j(\Omega,\eta)=\int_{\Omega}\lvert\nabla\eta (x)-A(x) \rvert^2dx+\int_{\Omega}\lvert\eta(x)-\eta_0(x) \rvert^2dx
\end{equation} 

 \begin{equation}\label{fonct1210}
J(K)=\int_{\Omega\backslash K}\lvert\nabla u_K (x)-A(x) \rvert^2dx+\int_{\Omega\backslash K}\lvert u_K(x)-u_0(x) \rvert^2dx
\end{equation}
 where we suppose that there exists an inclusion $K\subset \Omega$ inside the domain $\Omega$ i.e.,

and the problems
\begin{equation}\label{fonct12}
\min_{\Omega\in \mathcal{O}} J(\Omega)\end{equation}
\begin{eqnarray}\label{fonct121}
\min_{K\in \mathcal{O}} J(K)\end{eqnarray}
where $\eta$ is solution to
\begin{equation}\label{pge}\left\{\begin{array}{ccc}
-\Delta\eta-k^2\eta=f\,\,\text{in}\,\,\Omega\\[0.3cm]
\;\eta=0\,\,\,\,\text{on}\,\,\partial\Omega
\end{array}\right.\end{equation}
and
where $A\in L^2(\Omega\backslash K)$ and $u_0\in L^2(\Omega\backslash K)$ are given
and $u=u_K$  solution to the following Helmholtz equations.
\begin{equation}\label{pge1}
\begin{cases}
-\Delta u -k^2 u=g\,\,\text{in}\,\,\Omega\backslash K\\[0.4cm]
u=0\,\,\text{on}\,\,\partial K\\[0.4cm]
u=0\,\,\text{on}\,\,\partial\Omega\end{cases}
\end{equation}
where $k^2\neq 0$ stands for eigenvalue or not and $g\in L^2(\Omega\backslash K).$
and
where $\mathcal{O}$ is a set of admissible domains $\Omega\in\mathbb R^N,\,\,N\geq 2.$\\

\subsection{Existence of optimal shapes}
The aim of this subsection, is first, to get existence of an optimal shape  $\omega^*$ solution to the problem 
$$\min_{\omega\in\mathcal O}J(\omega),$$  under constraint (\ref{pge}) respectively (\ref{pge1}),
where $J$ is a functional  defined by (\ref{fonct0})  respectively (\ref{fonct1210}).
The principal objective of this paper is to prove that there exists a solution to every problem and to characterize the solutions. For this purpose, we shall be called to calculate for each problem, the shape derivative of the functional considered. 

Before going further, we first underline the existence of solutions to the problem (\ref{pge}) in the cases where $k^2$ is an eigenvalue or not. We have the first theorem.

\begin{theorem} 

1.  There exists an at most countable set $\Sigma\subset\mathbb R$ such that the boundary value problem  

\begin{equation}\label{0pge}\left\{\begin{array}{ccc}
-\Delta\eta-\lambda\eta= f\,\,\text{in}\,\, \Omega\\
\eta=0\,\,\text{on}\,\,\partial \Omega.
\end{array}\right.\end{equation}
 has a unique weak solution for each  $f\in L^2(\Omega)$ if and only if $\lambda\notin\Sigma.$\\

2. If $\Sigma$ is infinite, then $\Sigma=\{\lambda_k\}_{k=1}^{+\infty},$ the value of the non decreasing sequence with $\lambda_k\rightarrow+\infty.$

\end{theorem}
\begin{proof}(see \cite{Evans}).
\end{proof}

 \begin{remark}
 For these types of equations, in general, Lax Milgram's theorem does not allow to obtain the existence and uniqueness of the solution. It is for this reason, to be complete, that we review all the results dealing with this type of problem. The following result, that can be found in \cite{Evans}, uses the Fredholm alternative to get solution to the problem.
 \end{remark}
 
 \begin{theorem} The following statement holds:
 
 \begin{enumerate}
 \item for each $f\in L^2(\Omega)$ there exists a unique weak solution of the boundary-value problem (\ref{0pge})  or else

 \item there exists a weak solution $\eta\neq 0$ of the homogeneous problem
 
 \begin{equation}\label{0pge}\left\{\begin{array}{ccc}
-\Delta\eta-\lambda\eta= 0\,\,\text{in}\,\, \Omega\\
\eta=0\,\,\text{on}\,\,\partial \Omega.
\end{array}\right.\end{equation}
 
 \end{enumerate}
 
 \end{theorem}  
 \begin{proof} 
 The proof of this theorem is based on that given by Evans (Theorem 4 of chapter 6). The operator  $\Delta u+\lambda u$ that we use in this case is a special case of the one used in \cite{Evans}.  
 $$Lu=-\sum_{i,j=1}^n (a^{ij}(x)u_{x_i})_{x_j}+\sum_{i=1}^nb^i(x)u_{x_i}+c(x) u.$$
 It suffices to take $a_{ij}(x)=1\,\, \text{if}\,\,i=j$ and $a_{ij}=0\,\,\text{if}\,\,i\neq j,\,\,b^i=0$ and $c=-\lambda.$ The proof uses Fredholm's alternative.

 \end{proof}

In the case where $k^2$ is an eigenvalue, we have the boundary value problem
\begin{equation}\label{00pge}\left\{\begin{array}{ccc}
-\Delta\eta-k^2\eta= 0\,\,\text{in}\,\, \Omega\\
\eta=0\,\,\text{on}\,\,\partial \Omega.
\end{array}\right.\end{equation} 
This boundary value problem has a non trivial solution $\eta\neq 0$ if and only if $k^2\in\Sigma.$ In this case, $k^2$ is called and eigenvalue and $\eta$ the associated eigenfunction. Recall the following definitions of the capacity and the quasi-open of a set.

\begin{definition} 
Let $D$ be a bounded open of $\mathbb{R}^N$.\\
For any compact $K$ in $D$, let 
\begin{eqnarray*}
cap_D(K)=\inf \left\{ \int_D\lvert\nabla v\rvert;\;v\in C_0^{\infty}(D),\;v\geq 1\;\text{on}\;K \right\} < +\infty.
\end{eqnarray*}
For any open $\omega$ in $D$, we pose 
\begin{eqnarray*}
cap_D(\omega):=\sup\left\{cap_D(K);\;\;K\;\text{compact},\;K\subset\omega \right\}.
\end{eqnarray*}
If $E$ is any subset of $D$, let
\begin{eqnarray*}
cap_D(E):=\inf\left\{cap_D(\omega);\;\;\omega\;\text{open},\;E\subset\omega \right\}.
\end{eqnarray*}
\end{definition}
\begin{definition}
A subset $\Omega$ of $D$ is said to be quasi-open if there exists a decreasing sequence of open $\omega_n$ such that
\begin{eqnarray*}
\lim_{n\to \infty}\;cap(\omega_n)=0\\\forall\;n,\;\;\Omega\cup\omega_n\;\;\text{is open.}
\end{eqnarray*}
\end{definition}
We have the following existence theorem.

 \begin{theorem} Consider $\mathcal D\subset \mathbb R^N$ a bounded open set and let $A$ and $u_0\in L^2(\mathcal D)$  be given.  Let $J$ be the functional defined by (\ref{fonct0}) or (\ref{fonct1210}) where $\eta$ and $u_K$ are solution respectively to (\ref{pge}) and (\ref{pge1}). Then the shape optimization problem (\ref{fonct0}) or(\ref{fonct121}) admit a solution, where 
 $$\mathcal O=\{\Omega\,\,\text{quasi-open},\,\,\Omega\subset \mathcal D,\,\,\vert\Omega\vert\leq c\},\,\,c>0.$$
 
 \end{theorem}
 \begin{proof}Since the approach is the same in both the problems, we give the proof for the second shape optimization problem 
with the functional as follows
 $$J (\Omega)=j(\eta_\Omega)=\int_{\Omega}\lvert\nabla\eta (x)-A(x) \rvert^2dx+\int_{\Omega}\lvert\eta(x)-\eta_0(x) \rvert^2dx$$ 
 and $\eta_\Omega$ solution to
 \begin{equation}\label{11pge}\left\{\begin{array}{ccc}

-\Delta\eta_\Omega-k^2\eta_\Omega=f\,\,\text{in}\,\,\Omega\\[0.3cm]
\;\eta_\Omega=0\,\,\text{on}\,\,\partial\Omega.
\end{array}\right.\end{equation}
In the following, let  $F(x,s,z)= |z-A|^2 + |s-\eta_0|^2,$ then we have

$$F(x,\eta,\nabla \eta)=\lvert\nabla\eta (x)-A(x) \rvert^2+\lvert\eta(x)-\eta_0(x) \rvert^2.$$
For the proof we could use the general theory of $\gamma-$ convergence and weak $\gamma-$ convergence introduce by G. Buttazzo and G. Dalmaso \cite{buttazo-dalmazo} and G. Buttazzo and H. Shrivastava \cite{buttazo-shrivastava}.  We use directly a result of G. Buttazzo and H. Shrivastava \cite{buttazo-shrivastava}(see theorem 2.1). To apply this theorem, it suffices only to verify the following conditions of our integrand function $F:$\\
- $F(x,s,z)$ is measurable in $x$,  lower semicontinuous in $(s,z);$ and convex in $z,$\\
- there exists $c>0,\,a\in L^1(D),$ and $\alpha<\lambda_1(D)$ such that 
 $$c(\vert z\vert^2-\alpha\vert s\vert^2-a(x))\leq F(x,y,z)\,\,\text{for every}\,\,x,\,s,\,z,$$
 being $\lambda_1(D)$ the first Dirichlet eigenvalue;
 \\
 -$F(x,0,0)\geq0.$\\
 The first two conditions make it possible to obtain the semi-continuity and coercitivity of the functional.
At the end, we have $F(x,0,0)=\vert A\vert^2+\vert\eta_0\vert^2\geq0.$\\
Let $u_n$ a sequence such that 
 $$j(u_n)=\int_{\Omega}\lvert\nabla  u_n (x)- A(x) \rvert^2dx+\int_{\Omega}\lvert  u_n(x)-u_0(x) \rvert^2dx$$
Then, $j(u_n)=J(\Omega)>0\Longrightarrow \inf\{J(\omega),\,\,\omega\in\mathcal O\}>0.$ Let $\alpha=\inf\{J(\omega),\,\,\omega\in\mathcal O\}$, then there exists a minimizing sequence $(\Omega_n)\subset\mathcal O$ such that $J(\Omega_n)\rightarrow \alpha.$ Since $\Omega_n\in \mathcal O$ is bounded, there exists a $\Omega\in\mathcal O$ such that $\Omega_n$ $\gamma\rightarrow \Omega.$\\
 In an other hand, 
 Because of the fact that, the functional is lower semi-continue, we have directly $J(\Omega)\leq J(\Omega_{n})$ and $J(\Omega)\leq \inf \{J(\omega),\,\,\omega\in\mathcal O\}.$ Finally we have
 $$J(\Omega)=\inf \{J(\omega),\,\,\omega\in\mathcal O\}.$$
 \end{proof}

 On an other hand, let us consider the following eigenvalue or spectral problem 
 \begin{equation}\left\{\begin{array}{ccc}-\Delta \eta_k=\lambda_k(\Omega)u_k\,\,\text{in}\,\,\Omega\\[0.3cm]
 \eta_k\in H_0^1(\Omega),\end{array}\right.
 \end{equation} for $k=1,2,\ldots,.$ We have also the following result which can be found in \cite{bucur, thesis}.
 
 \begin{theorem} Let $\mathcal D\subset\mathbb R^N,\,\,N\geq 2$ be a smooth open set in $\mathbb R^n.$ Then the shape optimization problem
 $$\min\{\lambda_k(\Omega):\,\,\Omega\subset \mathcal D,\,\,,\vert\Omega\vert\leq c\}$$
  has a solution.
 \end{theorem}
 \begin{proof}
We have first to notice that, given a smooth bounded set $\mathcal D\subset\mathbb R^n$ and a quasi-open set $\Omega\subset D.$ The relative eigenvalues $\lambda_k(\Omega,\mathcal D)$ are variationally characterized as
$$\lambda_k(\Omega,\mathcal D)=\min_{S_k\subset H_0^1(\Omega,\mathcal D)}\max_{u\in S_k\backslash \{0\}} \frac{\int_{\Omega}\vert\nabla u\vert^2dx}{\int_\Omega u^2dx}   $$
where the minimum is over the $k$-dimensional subspaces $S_k$ of $H_0^1(\Omega,\mathcal D)$
and the Sobolev space
$H_0^1(\Omega,\mathcal D)$ is defined as
$$H_0^1(\Omega,\mathcal D)=\{u\in H^1(\mathcal D):\,\, u=0\,\,\text{ q. e}\,\,D\backslash\Omega\},$$
where we use the term quasi-everywhere in sense of the $H^1(\mathbb R^n)$-capacity.
 \end{proof}
 
 \section{Shape derivatives of shape functionals}
 Let $\Omega$ be an open set of $\mathbb R^N$ of $\mathcal C^2.$ 
 Consider a class of admissible sets $\mathcal O$ in $\mathbb R^N$ which is stable with respect to a familly of diffeomorphims $\phi_t$ defined by $\phi_t(\cdot)= (I+tV(\cdot))(\Omega)$  and $V: \mathbb R^N\rightarrow\mathbb R^N$ smooth vector fields.
 Let $\Omega_t$ the perturbed domain defined by 
 $\Omega_t=\phi_t(\Omega)$ where $\phi_t(x)=x+tV(x).$ The function $\phi$ satisfies the following hypothesis:
\begin{eqnarray}\label{fi}
\phi :t\in \rightarrow W^{1,\infty}(\mathbb{R}^N)\,\,\text{differentiable in}\,\,  0\,\,\text{with}\,\,\phi(0)=Id\,\,\text{and}\,\,\;\;\phi (0)=V.
\end{eqnarray}
Let us consider a shape function $J:\mathcal O\rightarrow\mathbb R$, we define the Eulerian derivative as
\begin{equation} 
DJ(\Omega, V)=\lim_{t\rightarrow 0^+}\frac{J(\Omega_t)-J(\Omega)}{t}.
\end{equation}
Commonly, this expression is called the shape derivative of the shape functional $J$ at $\Omega\in\mathcal O$ in the direction $V$ whenever the limits exists. Concerning the shape derivative, we refer to \cite{Delfour, Henrot, Simon, Sokolowski} and references therein.
We present now the main results concerning the derivative of the functionals considered. In functional (\ref{fonct0}) and  (\ref{fonct1210}), $\eta$ is solution to (\ref{pge}) and $u=\eta_K$ is solution to (\ref{pge1}). For the problems considered the following two cases will be considered:  $k^2 $ in the constraint Helmothtz equation is not an eigenvalue and  $ k^2 $ is an eigenvalue.
 \subsection{Shape derivative with  Helmholtz equation}
 The computation of the shape derivative of the shape functional depends strongly on the boundary condition of the domain. The same problem is considered but the only difference is located on the boundary condition of the constraint equation which we replace by an homogeneous Neumann condition.\\
 If $k ^ 2$ is not an eigenvalue, the calculation of the shape derivative is carried out quite easily by learning on the Hadamard's formulas \cite {Hadamard, Henrot}. The results are given in the following theorem.
 \begin{theorem}\label{nonpropre}
Let $\Omega\subset \mathbb{R}^N$, $N\geq 1$ be an open set of class $\mathcal{C}^2$ and $\Omega_t=\phi(t)\Omega$ where $\phi$ is given in (\ref{fi}). If $k^2\notin\Sigma$ then, the functional $J$ defined by (\ref{fonct0})  is differentiable at $t=0$ and we have
\begin{eqnarray*}
DJ(\Omega;V)_{t=0}=j'(0)=\int_{\partial\Omega}\left(\frac{\partial \eta}{\partial n}\frac{\partial p}{\partial n}-2(\frac{\partial \eta}{\partial n})^2+ \lvert\nabla\eta-A \rvert^2+\lvert\eta-\eta_0 \rvert^2\right)V\cdot nd\sigma
\end{eqnarray*}
where $p$ is the adjoint state
\begin{equation*}
\begin{cases}
-\Delta p-k^2p=2(\eta-\eta_0)-2\Delta\eta+2\nabla\cdot A\,\,\text{in}\,\,\Omega\\[0.3cm]
p=0\,\,\text{on}\,\,\partial\Omega.
\end{cases}
\end{equation*}
and $\eta$ solution to (\ref{pge}) with Dirichlet condition on $\partial\Omega$ and

\begin{eqnarray*}
DJ(\Omega;V)_{t=0}=j'(0)=\int_{\partial\Omega}\left(-\frac{\partial^2 \eta}{\partial n^2}\right)V\cdot np+\nabla\eta\cdot\nabla_\Gamma(V\cdot n)pd\sigma+\int_{\partial\Omega}\left( \lvert\nabla\eta-A \rvert^2+\lvert\eta-\eta_0 \rvert^2\right)V\cdot nd\sigma
\end{eqnarray*}
where $\nabla_\Gamma$ denotes the tangential gradient and $p$ is the adjoint state solution to
\begin{equation*}
\begin{cases}
-\Delta p-k^2p=2(\eta-\eta_0)-2\Delta\eta+2\nabla\cdot A\,\,\text{in}\,\,\Omega\\[0.3cm]
\frac{\partial p}{\partial n}=0\,\,\text{on}\,\,\partial\Omega
\end{cases}
\end{equation*}
and $\eta$ solution to (\ref{pge}) with Neumann condition on $\partial\Omega$.
\end{theorem}
\begin{proof}
Using Hadamard's formula, the shape derivative of the functional (\ref{fonct0})  at $t=0$ is given by
\begin{eqnarray}\label{d1}
j'(0)=\int_{\Omega}2\eta'(\eta-\eta_0)+2\nabla\eta'(\nabla\eta-A) dx+\int_{\partial\Omega} (\lvert\nabla\eta-A \rvert^2+\lvert\eta-\eta_0 \rvert^2)n\cdot Vd\sigma
\end{eqnarray}
where $\eta '$, the shape derivative for $\eta_t$ at $t=0$ is the unique solution to the following problem
\begin{equation*}\label{eta1}
\begin{cases}
-\Delta\eta'-k^2\eta'=0\,\,\text{in}\,\,\Omega\\[0.4cm] \eta'=-\frac{\partial \eta}{\partial n}(V\cdot n)\,\,\text{on}\,\,\partial\Omega.
\end{cases}
\end{equation*}
By introducing the adjoint state $p$ of (\ref{pge}) relative to the functional (\ref{fonct0}), we get
\begin{equation}\label{Addirichlet}
\begin{cases}
-\Delta p-k^2p=2(\eta-\eta_0)-2\Delta\eta+2\nabla\cdot A\,\,\text{in}\,\,\Omega\\[0.3cm] p=0\,\,\text{on}\,\,\partial\Omega.
\end{cases}
\end{equation}
Multiplying (\ref{Addirichlet}) by $\eta '$ and integrating over $\Omega$ we get
\begin{eqnarray*}
-\int_{\Omega}\eta ' \Delta pdx-\int_{\Omega}k^2p\eta ' dx=\int_{\Omega}2\eta '(\eta-\eta_0)dx-\int_{\Omega}2\eta ' \Delta\eta dx+2\int_{\Omega}\eta ' \nabla\cdot Adx
\end{eqnarray*}
Using Green's formula and equation (\ref{pge}), we have
\begin{eqnarray*}
\int_{\Omega}2\eta '(\eta-\eta_0)dx+\int_\Omega 2\nabla\eta\nabla\eta'dx-2\int_{\Omega}\nabla\eta ' \cdot Adx=-\int_{\partial\Omega}\eta' \frac{\partial p}{\partial n}d\sigma+\int_{\partial\Omega}2\eta ' \frac{\partial \eta}{\partial n}d\sigma
\end{eqnarray*}
Replacing this expression in (\ref{d1}) and the boundary condition of $\eta'$ we get finally
\begin{eqnarray*}
j'(0)=\int_{\partial\Omega}\left(\frac{\partial \eta}{\partial n}\frac{\partial p}{\partial n}-2(\frac{\partial \eta}{\partial n})^2+ \lvert\nabla\eta-A \rvert^2+\lvert\eta-\eta_0 \rvert^2\right)V\cdot nd\sigma
\end{eqnarray*}
giving the desired result.\\
In the case of Neumann boundary conditions, the only difference is the way to derive $\eta_t$ on the boundary. So $\eta'$ is solution  to the following boundary value problem
\begin{equation}\label{formderive}
\begin{cases}
-\Delta\eta'- k^2\eta'=0\,\,\text{in}\,\,\Omega\\[0.3cm] \frac{\partial \eta'}{\partial n}=\left(-\frac{\partial^2 \eta}{\partial n^2}\right) V\cdot n+\nabla \eta\cdot\nabla_\Gamma(V\cdot n)\,\,\text{on}\,\,\partial\Omega,  
\end{cases}
\end{equation}
The adjoint state $p$ is also given by
\begin{equation}\label{adjoint2}
\begin{cases}
-\Delta p-k^2p=2(\eta-\eta_0)-2\Delta\eta+2\nabla\cdot A\,\,\text{in}\,\,\Omega\\[0.3cm] \frac{\partial p}{\partial n} =0\,\,\text{on}\,\,\partial\Omega.  
\end{cases}
\end{equation}
Multiplying (\ref{adjoint2})  by $\eta'$ and integrating over $\Omega,$ we get
\begin{eqnarray*}
-\int_{\Omega}\eta ' \Delta pdx-\int_{\Omega}k^2p\eta ' dx=\int_{\Omega}2\eta '(\eta-\eta_0)dx-\int_{\Omega}2\eta ' \Delta\eta dx+2\int_{\Omega}\eta ' \nabla\cdot Adx.
\end{eqnarray*}
which, using (\ref{formderive}) and the homogeneous boundary condition of $\frac{\partial\eta}{\partial n}=\frac{\partial p}{\partial n}=0$ we get
\begin{eqnarray*}
j'(0)=\int_{\partial\Omega} \left(-\frac{\partial^2 \eta}{\partial n^2}\right) V\cdot np+\nabla \eta\cdot\nabla_\Gamma(V\cdot n) pd\sigma+\int_{\partial\Omega} (\lvert\nabla\eta-A \rvert^2+\lvert\eta-\eta_0 \rvert^2)n\cdot Vd\sigma,
\end{eqnarray*}
giving the desired result.
\end{proof}
\subsection{Shape derivative with $k^2$ is an eigenvalue}

In this case, we consider the same problem but the only difference is the fact that $k^2=k^2(\Omega)$ in the constraint equation (\ref{pge}) or (\ref{pge1}) is an eigenvalue.  What requires to calculate first the shape derivative of the eigenvalue 
$\lambda(\Omega_t)$ in the case of Dirichlet and Neumann condition at $t=0.$ The calculation of the shape derivatives of the first eigenvalue has been proposed by several authors including Henrot and Pierre \cite{Henrot}, \cite{He1},\cite{He2}, Caubet  et al. \cite{Caubet} and references therein. In the following, we give some results related to the shape derivative in the case of a simple eigenvalue. Later we will give the shape derivative of a multiple eigenvalue.
\begin{theorem}\label{proprsimpl}
Let $\Omega$ be $\mathcal C^2$ domain and $\Omega_t$ defined as follows. Consider the following  problems
\begin{equation}\label{Dirichlet1} \left\{\begin{array}{ccc}
-\Delta \eta_t-k^2(\Omega_t)\eta_t=0\,\,\text{in}\,\,\Omega_t\\[0.3cm]
\eta_t=0\,\,\text{on}\,\,\partial\Omega_t\end{array}\right.\end{equation} and
\begin{equation}\label{Neumann1}\left\{\begin{array}{ccc}
-\Delta \eta_t-k^2(\Omega_t)\eta_t=0\,\,\text{in}\,\,\Omega_t\\[0.3cm]
\frac{\partial\eta_t}{\partial n}=0\,\,\text{on}\,\,\partial\Omega_t.\end{array}\right.\end{equation} 
We suppose that $k^2$ is a simple eigenvalue.\\
Then, $t\rightarrow k^2(\Omega_t)$ is differentiable and the shape derivative  with respect to the vector fields $V$ is given by:
\begin{eqnarray}\label{val1}
(k^2)'(0)=-\int_{\partial\Omega}\left(\frac{\partial\eta}{\partial n}\right)^2 V\cdot n\,d\sigma
\end{eqnarray}
if $\eta$ is solution to (\ref{Dirichlet1}) at $t=0$ and 
\begin{eqnarray}\label{val2}
(k^2)'(0)=\int_{\partial\Omega}\left(\lvert\nabla\eta\rvert^2-k^2\eta^2\right)(V\cdot n)d\sigma
\end{eqnarray}
if $\eta$ is solution to (\ref{Neumann1}) at $t=0$.
\end{theorem}
\begin{proof} We only give the proof in the case of Neumann condition. The DIrichlet one is easier\\
We have already proven that $t\rightarrow \eta_t$ is differentiable at $t=0$ and the shape derivation is solution to
\begin{equation}\label{33}
\begin{cases}
-\Delta\eta ' =(k^2) ' \eta +k^2 \eta ' \,\,\text{in}\,\,\Omega \\[0.3cm]  \frac{\partial \eta'}{\partial n}=\left(-\frac{\partial^2 \eta}{\partial n^2}\right) V\cdot n+\nabla \eta\cdot\nabla_\Gamma(V\cdot n)\,\,\text{on}\,\,\partial\Omega.
\end{cases}
\end{equation}
From the normalization relation, we get 
\begin{equation}\label{34}
\int_{\partial\Omega}\eta^2(V\cdot n)d\sigma+2\int_\Omega \eta\eta'dx=0.
\end{equation}
From (\ref{33}) and (\ref{34}), we get finally
\begin{equation}\label{35}
\begin{cases}
-\Delta\eta ' =(k^2) ' \eta +k^2 \eta ' \,\,\text{in}\,\,\Omega \\[0.5cm]  \frac{\partial \eta'}{\partial n}=\left(-\frac{\partial^2 \eta}{\partial n^2}\right) V\cdot n+\nabla \eta\cdot\nabla_\Gamma(V\cdot n)\,\,\text{on}\,\,\partial\Omega \\[0.5cm]  \int_{\partial\Omega}\eta^2(V\cdot n)d\sigma+2\int_\Omega \eta\eta'dx=0.
\end{cases}
\end{equation}
On the other hand, multiplying the first equation of (\ref{Neumann1}) by $\eta_t$ and integrating over $\Omega$ we get
\begin{eqnarray*}
k^2(t)=\int_{\Omega_t}\lvert\nabla \eta_t\rvert^2dx,
\end{eqnarray*}
whose derivative  yields 
\begin{eqnarray*}
(k^2)'=\int_\Omega 2\nabla \eta\nabla\eta'dx+\int_{\partial\Omega}\lvert \nabla \eta\rvert^2(V\cdot n)d\sigma
\end{eqnarray*}
Integrating by parts and using Green formula in the first term after the preceding equality
\begin{eqnarray*}
(k^2)'(0)=\int_{\partial\Omega}\left(\lvert\nabla\eta\rvert^2-k^2\eta^2\right)(V\cdot n)d\sigma.
\end{eqnarray*}
\end{proof}

\begin{theorem}
Let $\Omega\subset \mathbb{R}^N$  be an open set of class $\mathcal{C}^2$ and $\Omega_t=\phi(t)\Omega$ where is defined in (\ref{fi}).  $\eta$ is solution to (\ref{pge}) with $k^2$ a simple eigenvalue. Then, the functional  $J$defined by (\ref{fonct0}) is differentiable and we have:\\
in the case of Dirichlet condition on $\partial\Omega$
\begin{eqnarray*}
j'(0)=\int_\Omega (k^2)'\eta pdx+\int_{\partial\Omega}\left(-\frac{\partial \eta}{\partial n}\frac{\partial p}{\partial n}-2(\frac{\partial \eta}{\partial n})^2+ \lvert\nabla\eta-A \rvert^2+\lvert\eta-\eta_0 \rvert^2\right)V\cdot nd\sigma
\end{eqnarray*}
where $p$ is solution to the adjoint state
\begin{equation*}
\begin{cases}
-\Delta p-k^2p=2(\eta-\eta_0)-2\Delta\eta+2\nabla\cdot A\,\,\text{in}\,\,\Omega\\[0.3cm]
\frac{\partial p}{\partial n} =0\,\,\text{on}\,\,\partial\Omega
\end{cases}
\end{equation*}
where $(k^2)'$ given by (\ref{val1}) and in the case of Neumann condition on $\partial\Omega$
\begin{eqnarray*}
j'(0)=\int_\Omega(k^2)'\eta p dx+\int_{\partial\Omega} \left(-\frac{\partial^2 \eta}{\partial n^2}\right) V\cdot np+\nabla \eta\cdot\nabla_\Gamma(V\cdot n) pd\sigma+\int_{\partial\Omega} (\lvert\nabla\eta-A \rvert^2+\lvert\eta-\eta_0 \rvert^2)n\cdot Vd\sigma.
\end{eqnarray*}
where $p$ is the solution to
\begin{equation*}
\begin{cases}
-\Delta p-k^2p=2(\eta-\eta_0)-2\Delta\eta+2\nabla\cdot A\,\,\text{in}\,\,\Omega\\[0.3cm]
\frac{\partial p}{\partial n} =0\,\,\text{on}\,\,\partial\Omega
\end{cases}
\end{equation*}
and $(k^2)'$ given by (\ref{val2}).

\end{theorem}
The proof of this theorem is essentially based on those of Theorems \ref{nonpropre} and \ref{proprsimpl}. It suffices simply to consider, in the calculation of the shape derivative of the functional, the shape derivative of the eigenvalue $k^2$ given in (\ref{val1}) or (\ref{val2}). \\  Let us now come to the shape derivative of a multiple eigenvalue is given in the following. In the case of a Dirichlet condition, the result is already given by A. Berger \cite{Amandine}. It is  also done in the paper written by B. Rousselet \cite{Rousselet}.

\begin{theorem}\textnormal{(Dirichlet condition)}\label{thm}\\
Let $\Omega$  be an open set of class $C^2.$ Let  $\lambda_k(\Omega)$ be an eigenvalue of multiplicity $p>2.$ Denotes $u_{k_1},u_{k_2},...,u_{k_p}$  an orthonormal family  of vectors associated with $\lambda_k$ for the scalar product in  $L^2$. Then $t\mapsto\lambda_k(\Omega_t)$ has a directional derivative at $t = 0$ which is one of the eigenvalues of the matrix $p\times p $ defined by 
\begin{eqnarray}
\mathcal{M}=(m_{i,j})\;\;avec\;\;m_{i,j} =\int_{\partial\Omega}\left(\frac{\partial \eta_{k_i}}{\partial n}\frac{\partial \eta_{k_j}}{\partial n}\right) V\cdot nds\;\;i,j=1,...,p
\end{eqnarray}
where $\frac{\partial \eta_{k_i}}{\partial n}$ is the normal derivative of the est la dérivée normale de la $k_i$-th proper function $\eta_{k_i}$ and $V\cdot n$  is normal displacement of the border induced by the fields $V$.
\end{theorem}
\begin{proof}
See \cite{Amandine, Rousselet}.
\end{proof}

\begin{theorem} \textnormal{(Neumann condition)}\\
Let $\Omega$  be an open set of class $\mathcal{C}^2$ and  $\lambda_k(\Omega)$ is an 
eigenvalue of multiplicity $p\geq 2.$ Denotes $u_{k_1},u_{k_2},\ldots,u_{k_p}$ an orthonormal family of associated vectors with $\lambda_k.$ Then  $t\mapsto\lambda_k(\Omega_t)$ has a directional derivative at $t=0$ which is the one of the eigenvalues of the matrix  $p\times p$ defined by
\begin{eqnarray}
\mathcal{M}=(m_{i,j})\,\,avec\,\,m_{i,j} =\int_{\partial\Omega}(\nabla\eta_{k_i}\cdot\nabla\eta_{k_j})d\sigma -k^2\int_{\partial\Omega}\eta_{k_i}\eta_{k_j} (V\cdot n)d\sigma\,\,i,j=1,\ldots,p.
\end{eqnarray}

\end{theorem}
\begin{proof}
The proof of this theorem is similar to that given in \cite{Amandine} for a Dirichlet condition. It is only a matter of defining the variational formulation in suitable spaces and then adapting the same ideas.
\end{proof}
\subsection{Shape optimization with obstacle problems}

In this precise case, we consider a bounded domain containing  a fluid described by Helmothz's equations and in which there exists an obstacle denoted  by $K.$ The shape of the obstacle is unknown, but we only have one of information, i.e. the variation of the flux on the boundary of the obstacle is null. We aim to characterize the optimal shape of $K.$ 
For this purpose, we shall use an approach of shape optimization.
We therefore begin by presenting the problem. Let $\Omega$ and $K$ two domains of $\mathbb R^2$ such that, $K\subset\Omega$ and $\partial\Omega\cap\partial K=\emptyset$. Let $J$ be the objective function defined by
\begin{eqnarray}\label{jobs}
J(K)=\int_{\Omega\backslash K}\lvert \nabla \eta- A\rvert^2 dx+\int_{\Omega\backslash K}\lvert \eta- \eta_0\rvert^2 dx.
\end{eqnarray}

where $\eta=\eta(K)$ is solution to
\begin{equation}\label{eobs1}
\begin{cases}
-\Delta \eta-k^2\eta=g\,\,\text{in}\,\, \Omega\backslash K\\[0.3cm] \frac{\partial \eta}{\partial n}=0\;\,\,\text{on}\,\,\partial K\\[0.3cm]  \eta=0\,\,\text{on}\,\,\partial\Omega,
\end{cases}
\end{equation}
where $\partial K$ is the boundary of $K.$ Let us defined the pertubed domain 
$K_t=\phi(t)(K)=(Id+tV)(K)$ and $\eta_t=\eta(K_t).$ The function $\phi$ is defined as previously. In $\Omega\backslash K_t$ the shape functional (\ref{jobs}) becomes 
\begin{eqnarray}
j(t)=J(K_t)=\int_{\Omega\backslash K_t}\lvert \nabla \eta_t- A\rvert^2 dx+\int_{\Omega\backslash K_t}\lvert \eta_t- \eta_0\rvert^2 dx
\end{eqnarray}
where $\eta_t$ is solution to
\begin{equation}\label{md6}
\begin{cases}
-\Delta \eta_t-k^2\eta_t=0\,\,\text{in}\,\,\Omega\backslash K_t\\[0.3cm] \frac{\partial \eta_t}{\partial n}=0\,\,\text{on}\,\,\partial K_t\\[0.3cm] \eta_t=0\,\,\text{on}\,\,\partial\Omega.
\end{cases}
\end{equation}
The objective of this section will also be to do the same work as previously by considering the position of the obstacle. We will therefore distinguish the various cases studied previously: the case of a simple eigenvalue and multiple eigenvalue  with a Dirichlet and Neumann conditions. Our first result is the following.
\begin{theorem}
Let $\Omega\subset\mathbb{R}^N$ be an open set of  class $\mathcal{C}^2$ and  $\Omega_t=\phi(t)\Omega$ where  $\phi$ is defined by (\ref{fi}). Let $\eta$ be the solution to (\ref{eobs1})  with $k^2\notin\Sigma$. Then the shape functional defined by (\ref{jobs}) is differentiable at $t=0$ and
\begin{eqnarray*}
DJ(\Omega;V)_{t=0}= j'(0)=
\int_{\partial K}\left(\lvert\nabla \eta- A\rvert^2+\lvert \eta- \eta_0\rvert^2  \right)(V\cdot n)d\sigma&+&
\int_{\partial K} p\left( \left(-\frac{\partial^2 \eta}{\partial n^2}\right) V\cdot n+\nabla \eta\cdot\nabla_{\Gamma}(V\cdot n)\right)d\sigma
\\&+&\int_{\partial\Omega}\left(\lvert\nabla \eta- A\rvert^2+\lvert\eta_0\rvert^2  \right)(V\cdot n)d\sigma.
\end{eqnarray*}
\end{theorem}
\begin{proof}
Denoting by $\eta '$ the shape derivative of  $t\mapsto \eta_t$ at $t=0,$ we have,
\begin{eqnarray*}
j'(0)=2\int_{\Omega\backslash K}\nabla \eta'(\nabla \eta- A)+\eta'( \eta- \eta_0) dx+\int_{\partial(\Omega\backslash K)}\left(\lvert\nabla \eta- A\rvert^2+\lvert \eta- \eta_0\rvert^2  \right)(V\cdot n)d\sigma
\end{eqnarray*}
where $\eta '$ is solution to
\begin{equation*}
\begin{cases}
-\Delta\eta'-k^2\eta'=0\,\,\text{in}\,\,\Omega\backslash K \\[0.4cm] \frac{\partial \eta'}{\partial n}=\left(-\frac{\partial^2 \eta}{\partial n^2}\right) V\cdot n+\nabla \eta\cdot\nabla_{\Gamma}(V\cdot n)\,\,\text{on}\,\,\partial K\\[0.4cm] \eta '=0\,\,\text{on}\,\,\partial\Omega.
\end{cases}
\end{equation*}
The adjoint problem is also given by
\begin{equation}\label{md7}
\begin{cases}
-\Delta p-k^2p=2(\eta-\eta_0)-2\Delta\eta+2\nabla\cdot A\,\,\text{in}\,\,\Omega\backslash K\\[0.3cm] \frac{\partial p}{\partial n} =0\,\,\text{on}\,\,\partial K \\[0.3cm] p=0\,\,\text{on}\,\,\partial\Omega 
\end{cases}
\end{equation}
Multiplying (\ref{md7}) by $\eta'$nand integrating over $\Omega\backslash K,$ we have
\begin{eqnarray*}
-\int_{\Omega\backslash K}\eta ' \Delta pdx-\int_{\Omega\backslash K}k^2p\eta ' dx=\int_{\Omega\backslash K}2\eta '(\eta-\eta_0)dx-\int_{\Omega\backslash K}2\eta ' \Delta\eta dx+2\int_{\Omega\backslash K}\eta ' \nabla\cdot Adx
\end{eqnarray*}
giving with the use all the hypotheses, the following formulas
\begin{eqnarray*}
\int_{\Omega\backslash K}2\eta '(\eta-\eta_0)dx+\int_{\Omega\backslash K} 2\nabla\eta\nabla\eta'dx-2\int_{\Omega\backslash K}\nabla\eta ' \cdot Adx=\int_{\partial K} p \frac{\partial \eta '}{\partial n}d\sigma.
\end{eqnarray*}
Replacing this expression in the formule derivation and taking into account the boundary condition of   $\frac{\partial \eta'}{\partial n}$ we get the desired result.
\end{proof}
In the case where $k^2$ is a simple eigenvalue, we have the following theorem.
\begin{theorem}
Let $\Omega\subset\mathbb{R}^N$ be an open set of  class $\mathcal{C}^2$ and  $\Omega_t=\phi(t)\Omega$ where  $\phi$ is defined by (\ref{fi}). Let $\eta$ be the solution to (\ref{eobs1})  with $k^2$  be  a simple eigenvalue. Then the functional  defined by (\ref{jobs}) is differentiable at $t=0$ and

\begin{eqnarray*}
j'(0)=\int_{\partial(\Omega\backslash K)}\left(\lvert\nabla \eta- A\rvert^2+\lvert \eta- \eta_0\rvert^2  \right)(V\cdot n)d\sigma-\int_{\Omega\backslash K}(k^2)' \eta pdx-\int_{\partial K} p \frac{\partial \eta '}{\partial n}d\sigma
\end{eqnarray*}
when $(k^2)'$ is given by
\begin{eqnarray*}
(k^2)'(0)=\int_{\partial(\Omega\backslash K)}\lvert\nabla\eta\rvert^2 (V\cdot n)d\sigma-k^2\int_{\partial K}\eta^2 V\cdot nd\sigma
\end{eqnarray*}
\end{theorem}
\begin{proof} 
The proof of this theorem is identical to the above theorem.
\end{proof}

If $k^2$ is a multiple eigenvalue, we have also the following result
\begin{theorem}
Let $\Omega\subset\mathbb{R}^N$ be an open set of  class $\mathcal{C}^2$ and $K\subset \Omega.$  Let $\lambda_k(K)$ an eigenvalue with multiplicity $p\geq 2$. Denotes by $\eta_{k_1},\eta_{k_2},...,\eta_{k_p}$ a family of orthonormal associated eigenvector with $\lambda_k.$ Then $t\mapsto\lambda_k(K_t)$ has a directional derivative at  $t=0$ which is one of an eigenvalue of the matrix $p\times p$ defined by
\begin{eqnarray}
\mathcal{M}=(m_{i,j})\;\;avec\;\;m_{i,j}=\int_{\partial(\Omega\backslash K)}(\nabla\eta_{k_i}\cdot\nabla\eta_{k_j})d\sigma -\lambda_k\int_{\partial K}\eta_{k_i}\eta_{k_j} (V\cdot n)d\sigma\;\;i,j=1,\ldots,p.
\end{eqnarray}
\end{theorem}
\begin{proof}
The proof of this theorem is almost identical to the one proposed for example in \cite{Amandine, Rousselet} . The idea is only to adapt the results developed in \cite{Amandine, Rousselet} with the considered spaces and bilinear form.
\end{proof}

\section{Topological Optimization }

The topological sensitivity analysis has been introduced by A. Friedman and M. S. Vogelus\cite{Friedman} in the context of shape inversion in electrostratics and by A. Schumacher \cite{Schumacher}, J. Sokolowski and A. Zochowski \cite{Sokolowski} for the minimization of the compliance. The principle is the following. Let us consider a cost function  $J(\Omega)= j(\Omega, u_\Omega)$ where $u_\Omega$ is the solution to a partial differential equation defined in the domain $\Omega\subset\mathbb R^n,\,\,n=2$ or $n=3,\,\,x_0\in\Omega$  and a fixed open and bounded subset $\omega$ of $\mathbb R^n$ containing the origin. The aim is to determine an asymptotic expansion of the criterion $J$ as follows
$$J(\Omega_\epsilon)-J(\Omega)=f(\epsilon)D_TJ(x_0)+o(f(\epsilon))$$
where $f(\epsilon)$ is a positive function tending to zero with $\epsilon$ small as small as we want. The function $D_TJ$ is called topological gradient or topological sensitivity. In this section, we are interested in the calculation of the topological derivative associated with the functional  defined in the next subsection under the constraint the Helmothz equation. We will be inspired by the work of Novotny and Sokolowski in \cite{NovotnySokolo} which lies the shape and topological optimization, Nazarov and Sokolowski \cite{Nazarov} and Masmoudi et al \cite{Masmoudi}.
\subsection{ Topological derivative with a perturbation on the source term}
\noindent
In this subsection, we consider the following functional
  \begin{eqnarray}\label{Jd1}
  J(\Omega)=\int_{\Omega}\lvert\nabla\eta-A \rvert^2dx+\int_{\Omega}\lvert\eta-\eta_0\rvert^2dx
  \end{eqnarray}
  where $\eta_0$ is the target function, assumed to be smooth and $A$ is given. The scalar field
$\eta$ is the solution of the following variation problem
  \begin{eqnarray}\label{eqd11}
  \begin{cases}
 \text{ Find}\,\,\eta\in H^1_0(\Omega)\,\,\text{such that}\,\,\\[0.3cm] \int_{\Omega}\nabla\eta\nabla vdx-\int_{\Omega}k^2\eta vdx=\int_{\Omega}fvdx\,\,\forall\;v\in H^1_0(\Omega),
  \end{cases}
   \end{eqnarray}
  where the function $f$ is assumed to be continuous.\\First, we calculate the topological derivative of the functional with a perturbation of the source term. We introduce a topological perturbation on the source term of the form as in \cite{NovotnySokolo}
  \begin{equation}\label{fepsilon}
  f_\epsilon(x)=
  \begin{cases}
  f(x)\;\;si\;\;x\in\Omega\backslash\overline{\omega_\epsilon}\\[0.2cm] \gamma f(x)\;\;si\;\;x\in\omega_\epsilon;
  \end{cases}
  \end{equation}
  where $\gamma\in (0,+\infty)$ is a constant parameter reflecting the physical and chemical properties of the material. The perturbed functional is defined by 
    \begin{eqnarray}\label{Jd11p}
    \psi(\chi_\epsilon)=J(\Omega_\epsilon)=\int_{\Omega}\lvert\nabla\eta_\epsilon-A \rvert^2dx+\int_{\Omega}\lvert\eta_\epsilon(x)-\eta_0(x) \rvert^2dx
    \end{eqnarray}
    where $\eta_\epsilon$  is solution to the variational problem
    \begin{equation}\label{eqd12}
    \begin{cases}
    \text{Find}\,\,\eta_\epsilon\in H^1_0(\Omega)\,\,\text{such that}\\[0.3cm] \int_{\Omega}\nabla\eta_\epsilon\nabla vdx-\int_{\Omega}k^2\eta_\epsilon vdx=\int_{\Omega}f_\epsilon vdx\,\,\forall\,v\in H^1_0(\Omega).
    \end{cases}
    \end{equation}
  We therefore need the following proposition
  
  \begin{prop1}\label{P1dt}
  Let $\psi(\chi_\epsilon(x_0))$ be the shape functional defined in $\Omega_\epsilon$ satisfying the following asymptotic development 
  \begin{equation}\label{dt} 
  \psi(\chi_\epsilon(x_0))=\psi(\chi)+ f(\epsilon)D_TJ(x_0)+\mathcal R(f(\epsilon)).
  \end{equation}
  It is assumed that the rest $\mathcal R(f(\epsilon))=o(f(\epsilon))$ additionally satisfies $\mathcal R'(f(\epsilon))\rightarrow 0$ when $\epsilon\rightarrow 0$. Then the topological derivative is written
  \begin{equation}\label{df-dt}
   D_TJ(x_0)=\lim_{\epsilon\rightarrow 0}\frac{1}{f'(\epsilon)}\frac{d}{d\epsilon}\psi(\chi_\epsilon(x_0)),
  \end{equation}
  where $\frac{d}{d\epsilon}\psi(\chi_\epsilon(x_0)),$ is the shape derivative of $\psi(\chi_\epsilon(x_0))$ with respect to the positive parameter $\epsilon$.
  \end{prop1}
  \begin{proof}
  Let us calculate the derivative of the expression (\ref{dt}) with respect to the real parameter $\epsilon$, i.e.
  \begin{eqnarray*}
  \frac{d}{d \epsilon}\psi(\chi_\epsilon(x_0))=f'(\epsilon)D_TJ(x_0)+\mathcal R'(f(\epsilon))f'(\epsilon).
  \end{eqnarray*}
  Dividing by $f'(\epsilon)$, we get
  \begin{eqnarray*}
  D_TJ(x_0)=\frac{1}{f'(\epsilon)}\frac{d}{d \epsilon}\psi(\chi_\epsilon(x_0))-\mathcal R'(f(\epsilon)).
  \end{eqnarray*}
  As $\lim_{\epsilon \to 0}R'(f(\epsilon)) =0$, by tending $\epsilon$ to $0$, we get the result.
  \end{proof}
  \begin{theorem}\label{theorem4.1}
 Let $\Omega\subset\mathbb R^n$ be an open set of class $\mathcal C^2$ and the function $J$ defined by (\ref{Jd11p}). Let $f_\epsilon$ defined by (\ref{fepsilon}).  Then, the topological derivative of the functional (\ref{Jd1}) is  given by
  \begin{eqnarray*}
  D_TJ(x_0)=(1-\gamma)f(x_0)p(x_0)
  \end{eqnarray*}
  where $p$ is the adjoint state, solution to the following variational equation 
    \begin{eqnarray*}
    \text{Find}\,\,p\in H^1_0(\Omega)\,\,\text{such that},\,\int_{\Omega}\nabla v\nabla p-k^2 v pdx=-2\int_{\Omega}(\nabla\eta-A)\nabla v+2(\eta-\eta_0)vdx,\;\forall\; v\in H^1_0(\Omega).
    \end{eqnarray*}
  \end{theorem}
  Before giving the proof of this theorem, we introduce this results which are usefull for the calculation of the topological derivative.
  \begin{lemma}\label{estimationleme}
Let $\eta$ and $\eta_\epsilon$ be the solutions of (\ref{eqd11}) and (\ref{eqd12})  respectively. Then, the following estimate holds%
  \begin{eqnarray}\label{ingalityestimation}
  \lVert\eta_\epsilon-\eta \rVert_{H^1(\Omega)}\leq C\epsilon
  \end{eqnarray}
  where $C$ is a positive constant independent of the small parameter $\epsilon.$ 
  \end{lemma}
  \begin{proof}
Let $a(\eta,v)$ be the bilinear formula
\begin{eqnarray*}
a(\eta,v)=\int_{\Omega}\nabla\eta\nabla vdx-\int_{\Omega}k^2\eta vdx
\end{eqnarray*}
 $a(\eta,v)$ is $H^1_0(\Omega)$ elliptical in the following sense: there are two constants $\alpha\in\mathbb{R}$ and $\beta >0$ such that 
      \begin{eqnarray}
      \lvert a(\eta,\eta)\rvert\geq \beta\lVert\eta\rVert^2_{H^1_0(\Omega)}-\alpha \lVert \eta\rVert^2_{L^2(\Omega)}
      \end{eqnarray}
 For this, just estimate $a(\eta, \eta)$, we have
       \begin{eqnarray*}
        a(\eta,v)&=&\int_{\Omega}\lvert\nabla\eta\rvert^2 dx-\int_{\Omega}k^2\lvert\eta\rvert^2 dx=\lVert\nabla\eta\rVert^2_{L^2(\Omega)}-k^2 \lVert \eta\rVert^2_{L^2(\Omega)}\\&=&\lVert\eta\rVert^2_{H^1_0(\Omega)}-k^2 \lVert \eta\rVert^2_{L^2(\Omega)}\\
        \end{eqnarray*}
So there exists $0<\beta\leq 1$ such that $a(\eta,v)$ is $H^1_0(\Omega)$ elliptical with $\alpha=k^2$. By taking $\tilde\eta_\epsilon=\eta_\epsilon-\eta$ we get from (\ref{eqd11}) and (\ref{eqd12}) the following result
      \begin{eqnarray*}
      \int_{\Omega}\nabla\tilde{\eta_\epsilon}\nabla vdx-\int_{\Omega}k^2\tilde{\eta_\epsilon} vdx&=&\int_{\Omega}f_\epsilon vdx-\int_{\Omega}f vdx \\ &=&\int_{\Omega\backslash \omega_\epsilon}f vdx+\gamma \int_{\omega_\epsilon}f vdx-\int_{\Omega}fvdx\\&=&\int_{\Omega}f vdx-(1-\gamma)\int_{\omega_\epsilon}f vdx-\int_{\Omega}f v\\&=&-(1-\gamma)\int_{\omega_\epsilon}f vdx
      \end{eqnarray*}
      This leads to
      \begin{eqnarray}
      \int_{\Omega}\nabla\tilde{\eta_\epsilon}\nabla vdx-\int_{\Omega}k^2\tilde{\eta_\epsilon} vdx=-(1-\gamma)\int_{\omega_\epsilon}f vdx
      \end{eqnarray}
     Since $\tilde{\eta}_\epsilon\in H^1_0(\Omega)$,  we can take $v=\tilde{\eta_\epsilon}$. This gives
      \begin{eqnarray*}
      \int_{\Omega}\lvert\nabla\tilde{\eta}_\epsilon\rvert^2dx-\int_{\Omega}k^2\lvert \tilde{\eta}_\epsilon\rvert^2 dx=-(1-\gamma)\int_{\omega_\epsilon}f\tilde{\eta}_\epsilon dx
      \end{eqnarray*}
      Using the fact that $a(\eta,\eta)$ is $H^1_0(\Omega)$ elliptical
      \begin{eqnarray*}
         \beta\lVert\tilde{\eta}_\epsilon\rVert^2_{H^1_0(\Omega)}-k^2 \lVert \tilde{\eta}_\epsilon\rVert^2_{L^2(\Omega)}\leq a(\tilde{\eta}_\epsilon,\tilde{\eta}_\epsilon)\leq \lvert 1-\gamma\rvert\int_{\omega_\epsilon}\lvert f\rvert\lvert\tilde{\eta}_\epsilon\rvert dx
          \end{eqnarray*}
      The Cauchy Schwarz inequality leads to:
      \begin{eqnarray*}
      \beta\lVert\tilde{\eta}_\epsilon\rVert^2_{H^1_0(\Omega)}-k^2 \lVert \tilde{\eta}_\epsilon\rVert^2_{L^2(\Omega)} &\leq&  \lvert 1-\gamma\rvert\lVert f\rVert_{L^2(\omega_\epsilon)}\lVert \tilde{\eta}_\epsilon\rVert_{L^2(\omega_\epsilon)}\\ &\leq& C_1\epsilon\lVert \tilde{\eta}_\epsilon\rVert_{L^2(\omega_\epsilon)}\\ &\leq& C_2\epsilon \lVert\tilde{\eta}_\epsilon\rVert_{H^1(\Omega)}.
      \end{eqnarray*}
    Alternatively, we can write
       \begin{eqnarray*}
        \beta\lVert\tilde{\eta}_\epsilon\rVert^2_{H^1_0(\Omega)}-k^2 \lVert \tilde{\eta}_\epsilon\rVert^2_{H^1(\Omega)} \leq C_2\epsilon \lVert\tilde{\eta}_\epsilon\rVert_{H^1(\Omega)}.
        \end{eqnarray*}
      As the norms $\lVert\cdot\rVert_{H^1_0(\Omega)}$ and $\lVert\cdot\rVert_{H^1(\Omega)}$ are equivalent, we can find $\beta_1>0$ such that
        \begin{eqnarray*}
           \beta_1\lVert\tilde{\eta}_\epsilon\rVert^2_{H^1(\Omega)}\leq \lVert\tilde{\eta}_\epsilon\rVert^2_{H^1_0(\Omega)}.
           \end{eqnarray*}
           Therefore, taking $\beta_2$ as the product of $\beta$ and $\beta_1$
           \begin{eqnarray*}
               \beta_2\lVert\tilde{\eta}_\epsilon\rVert^2_{H^1(\Omega)}-k^2 \lVert \tilde{\eta}_\epsilon\rVert^2_{H^1(\Omega)} \leq C_2\epsilon \lVert\tilde{\eta}_\epsilon\rVert_{H^1(\Omega)}.
               \end{eqnarray*}
   If $\beta_2-k^2>0$, then we have, for $C=\frac{C_2}{\beta_2-k^2}$
       \begin{eqnarray*}
                 \lVert\tilde{\eta}_\epsilon\rVert^2_{H^1(\Omega)} \leq C\epsilon.
       \end{eqnarray*}
   If $\beta_2-k^2 \leq 0$, we can find $M_1>0$ such that $M_1+\beta_2-k^2> 0$ and such that
       \begin{eqnarray*}
         (M_1+\beta_2-k^2)\lVert\tilde{\eta}_\epsilon\rVert^2_{H^1(\Omega)}\leq C_2\epsilon \lVert\tilde{\eta}_\epsilon\rVert_{H^1(\Omega)}
       \end{eqnarray*}
      and so by taking $C=\frac{C_2}{M_1+\beta_2-k^2}$, we obtain the desired result.\\
\end{proof}
The adjoint state $p_\epsilon$ is solution to
  \begin{equation}\label{eqd13}
  \begin{cases}
   \text{Find}\,\,p_\epsilon\in H^1_0(\Omega)\,\,\text{such that}\,\\[0.3cm] \int_{\Omega}\nabla v\nabla p_\epsilon dx-\int_{\Omega}k^2 v p_\epsilon dx=-2\int_{\Omega}(\nabla\eta_\epsilon-A)\nabla vdx-2\int_{\Omega}(\eta_\epsilon-\eta_0)vdx\,\,\forall\,v\in H^1_0(\Omega),
  \end{cases}
  \end{equation}
 and the adjoint state $p$ is solution to 
  \begin{equation}\label{eqd14}
  \begin{cases}
  \text{Find}\,\,p\in H^1_0(\Omega)\,\,\text{such that}\,\,\\[0.3cm] \int_{\Omega}\nabla v\nabla p dx-\int_{\Omega}k^2 v pdx=-2\int_{\Omega}(\nabla\eta-A)\nabla vdx-2\int_{\Omega}(\eta-\eta_0)vdx\;v\in H^1_0(\Omega).
  \end{cases}
  \end{equation}
  We have also the following lemma
  \begin{lemma}
 Let $p$ and $p_\epsilon$ be the solutions of (\ref{eqd14}) and (\ref{eqd13})  respectively. Then, the following estimate holds true
  \begin{eqnarray*}
  \lVert p_\epsilon-p \rVert_{H^1(\Omega)}\leq M\epsilon
  \end{eqnarray*}
  with $M$ is a positive constant not depending on the small parameter $\epsilon.$
  \end{lemma}
 \begin{proof}
 The proof of this lemma is identical to that of the lemma \ref{estimationleme}. We make the difference between (\ref{eqd13}) and (\ref{eqd14}) and use the inequality (\ref{ingalityestimation}).
    \end{proof}
\begin{proof}\,of Theorem\,\ref{theorem4.1}.\\
Let $\eta$ be the solution to (\ref{eqd11}) and $\eta_\epsilon$ the solution to (\ref{eqd12}), then $\eta_\epsilon-\eta$ is solution to
\begin{eqnarray}\label{eqd25}
      \int_{\Omega}\nabla(\eta_\epsilon-\eta)\nabla vdx-\int_{\Omega}k^2(\eta_\epsilon-\eta) vdx=-(1-\gamma)\int_{\omega_\epsilon}f vdx.
        \end{eqnarray}
On the other hand $\psi(\chi_\epsilon)-\psi(\chi)$ can be written as follows
    
\begin{eqnarray*}
\psi(\chi_\epsilon)-\psi(\chi)&=&\int_{\Omega}\lvert\nabla\eta_\epsilon-A \rvert^2dx+\int_{\Omega}\lvert\eta_\epsilon-\eta_0 \rvert^2dx-\int_{\Omega}\lVert\nabla\eta-A \rVert^2dx-\int_{\Omega}\lvert\eta-\eta_0\rvert^2dx\\&=&\int_{\Omega}\lvert\nabla\eta_\epsilon-\nabla\eta\rvert^2dx
+\int_{\Omega}\lvert\nabla\eta-A\rvert^2dx-\int_{\Omega}\lvert\nabla\eta-A \rvert^2dx+2\int_{\Omega}(\nabla\eta-A)\nabla(\eta_\epsilon-\eta)dx
\\&+&\int_{\Omega}\lvert\eta_\epsilon-\eta\rvert^2dx+\int_{\Omega}\lvert\eta-\eta_0 \rvert^2dx+2\int_{\Omega}(\eta_\epsilon-\eta)(\eta-\eta_0)dx-\int_{\Omega}\lvert\eta-\eta_0\rvert^2dx\\&=&2\int_{\Omega}(\nabla\eta-A)\nabla(\eta_\epsilon-\eta)dx+2\int_{\Omega}(\eta_\epsilon-\eta)(\eta-\eta_0)dx+\Vert\eta_\epsilon-\eta\Vert_{H^1(\Omega)}.
\end{eqnarray*}
Taking  $v=\eta_\epsilon-\eta$ as a test function in (\ref{eqd14}), we have 
 \begin{eqnarray}\label{testeta}
  \int_{\Omega}\nabla(\eta_\epsilon-\eta)\nabla pdx-\int_{\Omega}k^2(\eta_\epsilon-\eta) pdx=-2\int_{\Omega}(\nabla\eta-A)\nabla(\eta_\epsilon-\eta)dx-2\int_{\Omega}(\eta-\eta_0)(\eta_\epsilon-\eta)dx.
 \end{eqnarray}
In (\ref{eqd25}), if  $v=p,$ we get also
 \begin{eqnarray}\label{testp}
  \int_{\Omega}\nabla(\eta_\epsilon-\eta)\nabla pdx-\int_{\Omega}k^2(\eta_\epsilon-\eta) pdx=-(1-\gamma)\int_{\omega_\epsilon} fpdx.
 \end{eqnarray}
From equalities (\ref{testeta}) and (\ref{testp}), we get
\begin{eqnarray*}
2\int_{\Omega}(\nabla\eta-A)\nabla(\eta_\epsilon-\eta)dx+2\int_{\Omega}(\eta-\eta_0)(\eta_\epsilon-\eta)dx=(1-\gamma)\int_{\omega_\epsilon} fpdx,
 \end{eqnarray*}
therefore $\psi(\chi_\epsilon)-\psi(\chi)$ can be rewritten as
 \begin{eqnarray*}
 \psi(\chi_\epsilon)-\psi(\chi)&=&(1-\gamma)\int_{\omega_\epsilon} fpdx+\Vert\eta_\epsilon-\eta\Vert_{H^1(\Omega)}\\&=&(1-\gamma)\pi\epsilon^2f(x_0)p(x_0)+\Vert\eta_\epsilon-\eta\Vert_{H^1(\Omega)}+r(\epsilon)
  \end{eqnarray*}
where  $r(\epsilon)$ is given by
\begin{eqnarray*}
r(\epsilon)=(1-\gamma)\int_{\omega_\epsilon}\left( fp-f(x_0)p(x_0)\right) dx.
\end{eqnarray*}
If $r(\epsilon)=o(\epsilon^2)$, we conclude that, the topological derivative is given by
\begin{eqnarray*}
  D_TJ(x_0)=(1-\gamma)f(x_0)p(x_0).
\end{eqnarray*}
\end{proof}
\subsection{Topological derivative with  perturbations of the domain }
\noindent
Let $\Omega$ be an open of $\mathbb R^N,\,\,N\geq 2.$ Let $\omega_\epsilon,\,\,\epsilon>0$ be the domain defined by $\omega_\epsilon(x_0)=x_0+\epsilon\omega$ satisfying $\overline{\omega_\epsilon}\subset\Omega,$ and $x_0$ any point
of $\Omega$ and $\omega$ a fixed domain. In this part, the topological derivative of the shape functional $J$ associated with the Helmholtz equation is studied by considering a Dirichlet condition on the boundary of the hole $\omega_\epsilon$. In this case, the initial geometrical domain is topologically perturbed by the insertion of a hole. What concerns us here is the spectral problem, that is, when $k^2=\lambda$ is the first eigenvalue associated to the operator $-\Delta$, and therefore there is no longer a source term. Since the eigenvalue also depends on the domain, we will also look for its asymptotic expansion. Without loss of generality, we assume that $N=2$ and let $\eta_\epsilon$ be the solution of the following equation
\begin{equation}\label{pert1}\left\{\begin{array}{ccc}\Delta \eta_\epsilon+\lambda_\epsilon\eta_\epsilon= 0\,\,\text{in}\,\,\Omega_\epsilon\\
\eta_\epsilon=0\,\,\text{on}\,\,\partial\Omega\\
\eta_\epsilon=0\,\,\text{on}\,\,\partial\omega_\epsilon.\end{array}\right.\end{equation}
The objective of this section, is also, to get the topological derivative of the functional $J$ defined by (\ref{jlambda})
\begin{eqnarray}\label{jlambda}
J(\Omega)=\int_{\Omega}\lvert\nabla\eta (x)-A \rvert^2dx+\int_{\Omega}\lvert\eta(x)-\eta_0(x) \rvert^2dx
\end{eqnarray}
 where $\eta$ is solution to (\ref{pert1}), for $\epsilon=0$ ie.
 \begin{equation}\label{Eqd1}
 \begin{cases}
 \Delta\eta+\lambda\eta=0\;\;\text{in}\;\;\Omega\\[0.2cm] \;\;\;\;\;\; \eta=0\;\;\text{on} \;\;\partial\Omega
 \end{cases}
 \end{equation}
The variational problem associated to (\ref{pert1}) is given by
 \begin{equation}\label{eqd013}
  \begin{cases}
   \text{Find}\,\,\eta_\epsilon\in H^1_0(\Omega)\,\,\text{such that}\,\\[0.3cm] \int_{\Omega_\epsilon}\nabla \eta_\epsilon\nabla vdx-\int_{\Omega}\lambda_\epsilon \eta_\epsilon vdx=0\,\,\forall\,v\in H^1_0(\Omega)
  \end{cases}
  \end{equation}
In the unperturbed domain, the variational problem is given by
\begin{equation}\label{eqd0013}
  \begin{cases}
   \text{Find}\,\,\eta\in H^1_0(\Omega)\,\,\text{such that}\,\\[0.3cm] \int_{\Omega}\nabla \eta\nabla vdx-\int_{\Omega}\lambda \eta vdx=0\,\,\forall\,v\in H^1_0(\Omega)
  \end{cases}
  \end{equation}
Formulas (\ref{eqd013}) and (\ref{eqd0013}) allow us to pose
\begin{equation}
a_\epsilon(\eta_\epsilon,v)=\int_{\Omega_\epsilon}\nabla \eta_\epsilon\nabla vdx-\int_{\Omega_\epsilon}k^2 \eta_\epsilon v dx
\end{equation}
and 
\begin{equation}
a(\eta,v)=\int_{\Omega}\nabla \eta\nabla vdx-\int_{\Omega}k^2 \eta v dx.
\end{equation}
\subsubsection{Asymptotic expression of the eigenvalue $\lambda_\epsilon$}
Which requires, as in the following, to first prove a certain number of preliminary results, useful for the proof. We first prove.

We assume that the solution to the problem can be approximated as
\begin{eqnarray}
\eta_\epsilon(x)\sim\eta(x)+w_0(\epsilon^{-1}(x-x_0))+\epsilon \eta_1(x).
\end{eqnarray}
where the functions $\eta$ and $\eta_1$ depending on $x$ are regular approximation terms, i.e. functions defined in $\Omega$, $w_0$ depends on $\xi=\epsilon^{-1}(x-x_0)$ is a solution of an exterior boundary problem.\\
 We have the following theorem

\begin{theorem} 
Let $\epsilon>0,\,\,(\eta_\epsilon,\lambda_\epsilon)$ be the solution to (\ref{pert1}) and 
$(\eta,\lambda)$ be the  solution to (\ref{Eqd1}).  The we have, the following asymptotic expansions
\begin{equation}
\lambda_\epsilon= \lambda+4\pi\epsilon\eta(x_0)cap(\omega)+O(\epsilon^2)
\end{equation}
Moreover, we have
\begin{equation}\label{aproxilamda}
\vert\vert\eta_\epsilon-\eta\vert\vert_{H^1(\Omega)}= O(\epsilon^2).
\end{equation}
\end{theorem}
\begin{proof}
To prove the equality (\ref{aproxilamda}), we look for an expression of the eigenvalue $\lambda_\epsilon$ of the form
\begin{eqnarray}
\lambda_\epsilon\sim\lambda+\epsilon\lambda_1+\sum_{k=2}^{\infty}\epsilon^k\lambda_k,
\end{eqnarray}
where $\lambda$ is the first eigenvalue associated with the normalized eigenfunction $\eta$ of the problem (\ref{Eqd1}).\\ The first order approximation of the simple eigenvalue then takes the following form
\begin{eqnarray}\label{aslamb}
\lambda_\epsilon\sim\lambda+\epsilon\lambda_1.
\end{eqnarray}
Let us replace $\lambda_\epsilon$ and $\eta_\epsilon$ by their expression in the equation (\ref{pert1})
\begin{eqnarray*}
-\Delta(\eta(x)+w_0(\epsilon^{-1}(x-x_0))+\epsilon \eta_1(x))=(\lambda+\epsilon\lambda_1)(\eta(x)+w_0(\epsilon^{-1}(x-x_0))+\epsilon \eta_1(x)).
\end{eqnarray*}
This implies that 
\begin{eqnarray*}
-\Delta\eta(x)-\Delta w_0(\epsilon^{-1}(x-x_0))-\epsilon\Delta\eta_1(x))&=&\lambda\eta(x)+\lambda w_0(\epsilon^{-1}(x-x_0))+\epsilon\lambda \eta_1(x) + \epsilon\lambda_1\eta(x)\\[0.2cm] &\;\;&+\epsilon\lambda_1 w_0(\epsilon^{-1}(x-x_0))+\epsilon^2\lambda_1 \eta_1(x).
\end{eqnarray*}
Due to the fact that  $-\Delta\eta(x)=\lambda\eta(x)$ and $\Delta w_0(\epsilon^{-1}(x-x_0))=0$, we end up with
\begin{eqnarray*}
-\epsilon\Delta\eta_1(x))-\epsilon\lambda \eta_1(x)-\epsilon\lambda_1\eta(x)&=&\lambda w_0(\epsilon^{-1}(x-x_0))+\epsilon\lambda_1 w_0(\epsilon^{-1}(x-x_0))+\epsilon^2\lambda_1 \eta_1(x)\\[0.2cm] &=&-\epsilon\lambda\lvert x-x_0\rvert^{-1}\eta(x_0)cap(\omega)-\epsilon^2\lambda_1\lvert x-x_0\rvert^{-1}\eta(x_0)cap(\omega).
\end{eqnarray*}
Neglecting terms of order higher than  $\epsilon$, then we have the following equation
\begin{equation}
\begin{cases}
\Delta \eta_1(x) +\lambda\eta_1(x)+\lambda_1\eta(x)=\lambda\eta(x_0)cap(\omega)\lvert x-x_0\rvert^{-1},\;\;\;x\in\Omega\\[0.3cm]
\eta_1(x)=\eta(x_0)cap(\omega)\lvert x-x_0\rvert^{-1},\;\;\;x\in\partial\Omega
\end{cases}
\end{equation}
Let us multiply the equation by $\eta$ and integrate on $\Omega$
\begin{eqnarray*}
\int_\Omega\Delta \eta_1(x)\eta(x) +\lambda\eta_1(x)\eta(x)+\lambda_1\eta(x)^2 dx=\int_\Omega\lambda\eta(x_0)cap(\omega)\lvert x-x_0\rvert^{-1}\eta(x)dx
\end{eqnarray*}
Since $\int_\Omega\eta(x)^2dx=1$, we have:
\begin{eqnarray*}
\lambda_1&=&-\int_\Omega\Delta \eta_1(x)\eta(x)dx -\int_\Omega\lambda\eta_1(x)\eta(x)dx+\int_\Omega\lambda\eta(x_0)cap(\omega)\lvert x-x_0\rvert^{-1}\eta(x)dx\\[0.3cm] &=&-\int_\Omega\Delta \eta_1(x)\eta(x)dx +\int_\Omega\eta_1(x)\Delta\eta(x)dx+\int_\Omega\lambda\eta(x_0)cap(\omega)\lvert x-x_0\rvert^{-1}\eta(x)dx
\end{eqnarray*}
Using Green's formula, we have
\begin{eqnarray*}
\lambda_1&=&-\int_\Omega\Delta \eta_1(x)\eta(x)dx +\int_\Omega\eta_1(x)\frac{\partial\eta(x)}{\partial n}ds-\int_\Omega\nabla\eta_1(x)\nabla\eta(x) dx+\int_\Omega\lambda\eta(x_0)cap(\omega)\lvert x-x_0\rvert^{-1}dx\\[0.2cm]
&=&-\int_\Omega\Delta \eta_1(x)\eta(x)dx +\int_{\partial\Omega}\eta_1(x)\frac{\partial\eta(x)}{\partial n}ds+\int_\Omega\Delta\eta_1(x)\eta(x) dx-\int_{\partial\Omega}\eta(x)\frac{\partial\eta_1(x)}{\partial n}ds\\[0.2cm] &+&\int_\Omega\lambda\eta(x_0)cap(\omega)\lvert x-x_0\rvert^{-1}\eta(x)dx\\[0.2cm]
&=&\int_{\partial\Omega}\eta_1(x)\frac{\partial\eta(x)}{\partial n}ds-\int_{\partial\Omega}\eta(x)\frac{\partial\eta_1(x)}{\partial n}ds+\int_\Omega\lambda\eta(x_0)cap(\omega)\lvert x-x_0\rvert^{-1}\eta(x) dx
\end{eqnarray*}
Taking into account the fact that $\eta$ is null on the boundary of $\Omega$ and replacing $\eta_1$ by its expression on the edge we have
\begin{eqnarray*}
\lambda_1&=&\eta(x_0)cap(\omega)\left\lbrace\int_{\partial\Omega}\frac{\partial\eta(x)}{\partial n}\lvert x-x_0\rvert^{-1} ds+\int_\Omega\lambda\lvert x-x_0\rvert^{-1}\eta(x) dx \right\rbrace\\[0.2cm] &=&\eta(x_0)cap(\omega)\left\lbrace\int_{\partial\Omega}\frac{\partial\eta(x)}{\partial n}\lvert x-x_0\rvert^{-1} ds-\int_\Omega\lvert x-x_0\rvert^{-1}\Delta\eta(x) dx \right\rbrace
\end{eqnarray*}
Using Green's formula we get
\begin{eqnarray*}
\lambda_1 &=&\eta(x_0)cap(\omega)\left\lbrace \int_{\partial\Omega}\frac{\partial\eta(x)}{\partial n}\lvert x-x_0\rvert^{-1} ds-\int_{\partial\Omega}\frac{\partial\eta(x)}{\partial n}\lvert x-x_0\rvert^{-1} ds+\int_\Omega\nabla(\lvert x-x_0\rvert^{-1})\nabla\eta(x) dx \right\rbrace\\[0.2cm] &=&\eta(x_0)cap(\omega)\int_\Omega\nabla(\lvert x-x_0\rvert^{-1})\nabla\eta(x) dx \\[0.2cm] &=&
\eta(x_0)cap(\omega)\left\lbrace \int_{\partial\Omega}\frac{\partial(\lvert x-x_0\rvert^{-1})}{\partial n}\eta(x) ds-\int_\Omega\Delta(\lvert x-x_0\rvert^{-1})\eta(x) dx \right\rbrace\\[0.2cm] &=&
-\eta(x_0)cap(\omega)\int_\Omega\Delta(\lvert x-x_0\rvert^{-1})\eta(x) dx\\[0.2cm] &=&4\pi\eta(x_0)cap(\omega)\int_\Omega\delta(x-x_0)\eta(x) dx\\[0.2cm] &=&4\pi\eta(x_0)cap(\omega)\int_\Omega\delta(y)\eta(y+x_0) dy\\[0.3cm] &=&4\pi\eta(x_0)^2cap(\omega)
\end{eqnarray*}
Therefore, we obtain
\begin{eqnarray*}
\lambda_\epsilon=\lambda+4\epsilon\pi\eta(x_0)^2cap(\omega)+O(\epsilon^2)
\end{eqnarray*}
where $\delta$ is the Dirac mass at the origin
\end{proof}

\subsubsection{Asymptotic expansion of the functional}
In the following, we prove the following theorem establishing the topological derivative of the functional
\begin{theorem}
Let $J$ be the functional given by
\begin{eqnarray*}
J(\Omega)=\int_{\Omega}\lvert\nabla\eta (x)-A(x_0) \rvert^2dx+\int_{\Omega}\lvert\eta(x)-\eta_0(x) \rvert^2dx
\end{eqnarray*}
then the topological derivative of $J$ at point $x_0$ is given by
\begin{eqnarray*}
DT(x_0)=mes(\omega)(\nabla\eta(x_0)\nabla p(x_0)-\lambda\eta(x_0) p(x_0)-\lvert \nabla\eta(x_0)-A(x_0)\rvert^2-\lvert \eta(x_0)-\eta_0(x_0)\rvert^2)
\end{eqnarray*}
where $\eta$ is the solution of the Helmholtz equation (\ref{Eqd1}) and $p$ the solution of the adjoint problem
\begin{eqnarray*}
\text{Find}\;p\in H^1_0(\Omega),\;-\int_\Omega\nabla p\nabla vdx+\lambda\int_\Omega pvdx=2\int_\Omega (\nabla\eta-A)\nabla vdx+2\int_\Omega (\eta-\eta_0)vdx\;\;\forall\; v\in H^1_0(\Omega).
\end{eqnarray*}
\end{theorem}
\begin{proof}
Let's start by calculating the variation of the cost function  $J(\Omega_\epsilon)-J(\Omega)$.
\begin{eqnarray*}
J(\Omega_\epsilon)-J(\Omega)&=&\int_{\Omega_\epsilon}\lvert \nabla\eta_\epsilon-A\rvert^2dx+\int_{\Omega_\epsilon}\lvert \eta_\epsilon-\eta_0\rvert^2dx-\int_{\Omega}\lvert \nabla\eta-A\rvert^2dx-\int_{\Omega}\lvert \eta-\eta_0\rvert^2dx\\
&=&\int_{\Omega_\epsilon}\lvert \nabla\eta_\epsilon-A\rvert^2-\lvert \nabla\eta-A\rvert^2dx+\int_{\Omega_\epsilon}(\lvert \eta_\epsilon-\eta_0\rvert^2-\lvert\eta-\eta_0\rvert^2)dx\\&-&\int_{\omega_\epsilon}\lvert \nabla\eta-A\rvert^2+\lvert \eta-\eta_0\rvert^2dx.
\end{eqnarray*}
Let's calculate the following integral $\int_{\omega_\epsilon}\lvert \nabla\eta-A\rvert^2+\lvert \eta-\eta_0\rvert^2dx$. Without loss of generality, we also assume that $x_0=0$. Thus $\omega_\epsilon$ is written, $\omega_\epsilon=\epsilon\omega$ and thus $x$ in $\omega_\epsilon$ means that  $x=\epsilon y$, where $y\in\omega$. Now let us put 
\begin{eqnarray*}
B_1=\int_{\omega_\epsilon}\lvert \nabla\eta(x)-A(x)\rvert^2+\lvert \eta(x)-\eta_0(x)\rvert^2dx.
\end{eqnarray*}
Then $B_1$ can be rewritten as follows
\begin{eqnarray*}
B_1&=&\int_{\omega_\epsilon}\lvert \nabla\eta(x)-A(x)\rvert^2+\lvert \eta(x)-\eta_0(x)\rvert^2dx\\
&=&\int_{\omega_\epsilon}\lvert \nabla\eta(0)-A(0)\rvert^2+\lvert \eta(0)-\eta_0(0)\rvert^2dx+\int_{\omega_\epsilon}\lvert \nabla\eta(x)-A(x)\rvert^2-\lvert \nabla\eta(0)-A(0)\rvert^2dx\\&+&\int_{\omega_\epsilon}\lvert \eta(x)-\eta_0(x)\rvert^2-\lvert \eta(0)-\eta_0(0)\rvert^2dx\\&=&\epsilon^2 mes(\omega)(\lvert \nabla\eta(0)-A(0)\rvert^2+\lvert \eta(0)-\eta_0(0)\rvert^2)+\int_{\omega_\epsilon}\lvert \nabla\eta(x)-A(x)\rvert^2-\lvert \nabla\eta(0)-A(0)\rvert^2dx\\&+&\int_{\omega_\epsilon}\lvert \eta(x)-\eta_0(x)\rvert^2-\lvert \eta(0)-\eta_0(0)\rvert^2dx.\\
\end{eqnarray*}
By a change of variables we have
\begin{eqnarray*}
&\,&\int_{\omega_\epsilon}\lvert \nabla\eta(x)-A(x)\rvert^2-\lvert \nabla\eta(0)-A(0)\rvert^2dx=\epsilon^2\int_{\omega}\lvert \nabla\eta(\epsilon y)-A(\epsilon y)\rvert^2-\lvert \nabla\eta(0)-A(0)\rvert^2dy,\\
&\,&\int_{\omega_\epsilon}\lvert \eta(x)-\eta_0(x)\rvert^2-\lvert \eta(0)-\eta_0(0)\rvert^2dx=\epsilon^2\int_{\omega}\lvert \eta(\epsilon y)-\eta_0(\epsilon y)\rvert^2-\lvert \eta(0)-\eta_0(0)\rvert^2dy.
\end{eqnarray*}
Thanks to the Taylor expansion, we obtain
\begin{eqnarray*}
&\;&\int_{\omega}\lvert \nabla\eta(\epsilon y)-A(\epsilon y)\rvert^2-\lvert \nabla\eta(0)-A(0)\rvert^2dy=O(\epsilon),\\
&\;&\int_{\omega}\lvert \eta(\epsilon y)-\eta_0(\epsilon y)\rvert^2-\lvert \eta(0)-\eta_0(0)\rvert^2dy=O(\epsilon).
\end{eqnarray*}
Hence
\begin{eqnarray*}
B_1=\epsilon^2 mes(\omega)\left( \lvert \nabla\eta(0)-A(0)\rvert^2+\lvert \eta(0)-\eta_0(0)\rvert^2\right)+o(\epsilon^3).
\end{eqnarray*}
In the same way we have
\begin{eqnarray*}
B_2&=&\int_{\Omega_\epsilon}\lvert\nabla\eta_\epsilon\rvert^2+\lvert A\rvert^2-2\langle\nabla\eta_\epsilon, A\rangle-\lvert\nabla\eta\rvert^2-\lvert A\rvert^2+2\langle\nabla\eta, A\rangle dx\\[0.2cm] &+&\int_{\Omega_\epsilon}\lvert\eta_\epsilon\rvert^2+\lvert\eta_0\rvert^2-2\langle\eta_\epsilon,\eta_0\rangle-\lvert\eta\rvert^2-\lvert\eta_0\rvert^2+2\langle\eta,\eta_0\rangle dx\\[0.2cm] 
&=&\int_{\Omega_\epsilon}\lvert\nabla\eta_\epsilon\rvert^2-\lvert\nabla\eta\rvert^2-2\langle\nabla\eta_\epsilon-\nabla\eta, A\rangle dx+\int_{\Omega_\epsilon}\lvert\eta_\epsilon\rvert^2-\lvert\eta\rvert^2-2\langle\eta_\epsilon-\eta,\eta_0\rangle dx.
\end{eqnarray*}
Moreover we have
\begin{eqnarray*}
\int_{\Omega_\epsilon}\lvert\nabla\eta_\epsilon\rvert^2-\lvert\nabla\eta\rvert^2dx=\int_{\Omega_\epsilon}\nabla(\eta_\epsilon-\eta)\cdot\nabla(\eta_\epsilon+\eta)dx,
\end{eqnarray*}
we have the following inequality
\begin{eqnarray*}
\left|\int_{\Omega_\epsilon}\lvert\nabla\eta_\epsilon\rvert^2-\lvert\nabla\eta\rvert^2dx\right| &=&\left| \int_{\Omega_\epsilon}\nabla(\eta_\epsilon-\eta)\cdot\nabla(\eta_\epsilon+\eta)dx \right|\\&\leq&\int_{\Omega}\lVert\nabla(\eta_\epsilon-\eta)\rVert_{L^2(\Omega)}\lVert\nabla(\eta_\epsilon+\eta)\rVert_{L^2(\Omega)}dx\\
&\leq &mes(\Omega)\lVert\nabla(\eta_\epsilon-\eta)\rVert_{L^2(\Omega)}\lVert\nabla(\eta_\epsilon+\eta)\rVert_{L^2(\Omega)}.
\end{eqnarray*}
Thanks to Poincare's inequality, there exist constants $C_1 >0$ and $C_2>0$ such that:
\begin{eqnarray*}
\lVert\nabla(\eta_\epsilon-\eta)\rVert_{L^2(\Omega)}\leq C_1\lVert\eta_\epsilon-\eta\rVert_{H^1_0(\Omega)}\\[0.2cm]
\lVert\nabla(\eta_\epsilon+\eta)\rVert_{L^2(\Omega)}\leq C_2\lVert\eta_\epsilon+\eta\rVert_{H^1_0(\Omega)}.
\end{eqnarray*}
Therefore there exists a constant $C>0$ such that:
\begin{eqnarray*}
\lVert\nabla(\eta_\epsilon-\eta)\rVert_{L^2(\Omega)}\lVert\nabla(\eta_\epsilon+\eta)\rVert_{L^2(\Omega)}\leq C\lVert\eta_\epsilon-\eta\rVert_{H^1_0(\Omega)}\lVert\eta_\epsilon+\eta\rVert_{H^1_0(\Omega)}
\end{eqnarray*}
and it follows that
\begin{eqnarray*}
\left|\int_{\Omega_\epsilon}\nabla(\eta_\epsilon-\eta)\cdot\nabla(\eta_\epsilon+\eta)dx \right|\leq C\lVert\eta_\epsilon-\eta\rVert_{H^1_0(\Omega)}\lVert\eta_\epsilon+\eta\rVert_{H^1_0(\Omega)}.
\end{eqnarray*}
In the same way we have:
\begin{eqnarray*}
\int_{\Omega_\epsilon}\lvert\eta_\epsilon\rvert^2-\lvert\eta\rvert^2dx&=&\int_{\Omega_\epsilon}(\eta_\epsilon-\eta)(\eta_\epsilon+\eta)dx.
\end{eqnarray*}
This gives
\begin{eqnarray*}
\left|\int_{\Omega_\epsilon}(\eta_\epsilon-\eta)(\eta_\epsilon+\eta)dx\right|&\leq&\int_{\Omega_\epsilon}\lvert\eta_\epsilon-\eta\rvert\lvert\eta_\epsilon+\eta\lvert dx\\&\leq&\int_{\Omega}\lvert\eta_\epsilon-\eta\rvert\lvert\eta_\epsilon+\eta\lvert dx\leq mes(\Omega)\lVert\eta_\epsilon-\eta\rVert_{H^1_0(\Omega)}\lVert\eta_\epsilon+\eta\rVert_{H^1_0(\Omega)}.
\end{eqnarray*}
Since $\lVert\eta_\epsilon-\eta\rVert_{H^1_0(\Omega)}=O(\epsilon^2)$, we get
\begin{eqnarray*}
J(\Omega_\epsilon)-J(\Omega)&=&-\int_{\Omega_\epsilon}2\langle\nabla\eta_\epsilon-\nabla\eta, A\rangle dx-\int_{\Omega_\epsilon}2\langle\eta_\epsilon-\eta,\eta_0\rangle dx \\&=&-\epsilon^2 mes(\omega)\left(\lvert \nabla\eta(0)-A(0)\rvert^2+\lvert \eta(0)-\eta_0(0)\rvert^2\right) +o(\epsilon^3).
\end{eqnarray*}
So we finally get
\begin{eqnarray*}
\frac{J(\Omega_\epsilon)-J(\Omega)}{\epsilon^2}\rightarrow -mes(\omega)\left(\lvert \nabla\eta(0)-A\rvert^2+\lvert \eta(0)-\eta_0(0)\rvert^2\right)
\end{eqnarray*}
Now we introduce the Lagrangian $L(\Omega,\eta,p):=F(\Omega)$:
\begin{eqnarray*}
F(\Omega)=J(\Omega)-\int_{\Omega}\nabla\eta\nabla pdx+\lambda\int_{\Omega}\eta pdx
\end{eqnarray*}
\begin{eqnarray*}
F(\Omega_\epsilon)-F(\Omega)&=&J(\Omega_\epsilon)-J(\Omega)-\int_{\Omega_\epsilon}\nabla\eta_\epsilon\nabla pdx+\lambda_\epsilon\int_{\Omega_\epsilon}\eta_\epsilon pdx+\int_{\Omega}\nabla\eta\nabla pdx-\lambda\int_{\Omega}\eta pdx\\&=& J(\Omega_\epsilon)-J(\Omega)-\int_{\Omega_\epsilon}\nabla\eta_\epsilon\nabla pdx+\lambda_\epsilon\int_{\Omega_\epsilon}\eta_\epsilon pdx+\int_{\Omega_\epsilon}\nabla\eta\nabla pdx-\lambda\int_{\Omega_\epsilon}\eta pdx\\&\;&+\int_{\omega_\epsilon}\nabla\eta\nabla pdx-\lambda\int_{\omega_\epsilon}\eta pdx\\&=& J(\Omega_\epsilon)-J(\Omega)-\int_{\Omega_\epsilon}(\nabla\eta_\epsilon-\nabla\eta)\nabla pdx+\int_{\Omega_\epsilon}(\lambda_\epsilon\eta_\epsilon-\lambda\eta)pdx\\&\;&+\int_{\omega_\epsilon}\nabla\eta\nabla pdx-\lambda\int_{\omega_\epsilon}\eta pdx.
\end{eqnarray*}
Replacing $\lambda_\epsilon$ by its expression, we have:
\begin{eqnarray*}
\int_{\Omega_\epsilon}(\lambda_\epsilon\eta_\epsilon-\lambda\eta)vdx&=&\int_{\Omega_\epsilon}((\lambda+4\pi\epsilon\eta(0)^2cap(\omega)+O(\epsilon^2))\eta_\epsilon-\lambda\eta)pdx\\
&=&\int_{\Omega_\epsilon}\left(\lambda\eta_\epsilon+4\pi\epsilon\eta(0)^2cap(\omega)\eta_\epsilon+O(\epsilon^2)\eta_\epsilon-\lambda\eta\right) pdx\\
&=&\int_{\Omega_\epsilon}\lambda\left(\eta_\epsilon-\eta\right)p dx
 +4\pi\epsilon\eta(0)^2cap(\omega)\int_{\Omega_\epsilon}\eta_\epsilon pdx+\int_{\Omega_\epsilon}O(\epsilon^2)\eta_\epsilon pdx
\end{eqnarray*}
Using both the triangular and Cauchy Schwarz inequalities , we get
\begin{eqnarray*}
\left|\int_{\Omega_\epsilon}(\lambda_\epsilon\eta_\epsilon-\lambda\eta)vdx \right|&\leq&
\lambda\lVert \eta_\epsilon-\eta\rVert\lVert p\lVert+4\pi\epsilon\eta(0)^2cap(\omega)\lVert\eta_\epsilon\lVert \lVert p\lVert +O(\epsilon^2)\lVert\eta_\epsilon\lVert \lVert p\lVert\rightarrow 0
\end{eqnarray*}
 \begin{eqnarray*}
 \int_{\omega_\epsilon}\nabla\eta\nabla pdx-\lambda\int_{\omega_\epsilon}\eta pdx&=&\epsilon^2\int_{\omega}\nabla\eta(\epsilon y)\nabla p(\epsilon y)dy-\lambda\epsilon^2\int_{\omega}\eta(\epsilon y) p(\epsilon y)dy\\&=&\epsilon^2\int_{\omega}\nabla\eta(0)\nabla p(0)dy-\lambda\epsilon^2\int_{\omega}\eta(0) p(0)dy+o(\epsilon^3)\\&=&\epsilon^2mes(\omega)(\nabla\eta(0)\nabla p(0)-\lambda\eta(0) p(0))+o(\epsilon^3).
 \end{eqnarray*}
So we get
 \begin{eqnarray*}
  \frac{1}{\epsilon^2}\left(\int_{\omega_\epsilon}\nabla\eta\nabla pdx-\lambda\int_{\omega_\epsilon}\eta pdx\right)\rightarrow mes(\omega)\left(\nabla\eta(0)\nabla p(0)-\lambda\eta(0) p(0)\right).
  \end{eqnarray*}
Thus the topological derivative of the functional $J$ at point $x_0=0$ is given by
\begin{eqnarray*}
DT(x_0)=mes(\omega)(\nabla\eta(x_0)\nabla p(x_0)-\lambda\eta(x_0)p(x_0)-\lvert \nabla\eta(x_0)-A(x_0)\rvert^2-\lvert \eta(x_0)-\eta_0(x_0)\rvert^2)
\end{eqnarray*}
where $\eta$ is the solution of the Helmholtz equation (\ref{Eqd1}) and $p$ the solution of the adjoint problem
\begin{eqnarray*}
\text{Find}\;p\in H^1_0(\Omega),\;-\int_\Omega\nabla p\nabla vdx+\lambda\int_\Omega pvdx=2\int_\Omega (\nabla\eta-A)\nabla vdx+2\int_\Omega (\eta-\eta_0)vdx\;\;\forall\; v\in H^1_0(\Omega).
\end{eqnarray*}
\end{proof}
\newline

\noindent
ACKNOWLEDGEMENT\\
\newline
The authors would like to thank Volker Schulz ( University Trier, Trier, Germany) and Luka Schlegel (University Trier, Trier, Germany) for helpful and interesting discussions within the project Shape Optimization Mitigating Coastal Erosion (SOMICE).

\end{document}